\documentclass[11pt, a4]{amsart}

\usepackage{fancyhdr}
\usepackage[all]{xy}

\usepackage{latexsym,amsfonts,amscd,yhmath,epsfig,epic}

\usepackage{ccaption}
\changecaptionwidth
\captionwidth{5.6in}

\usepackage{amsmath,amsthm,amssymb,mathrsfs,amscd,amsthm,amscd,
amsfonts,graphicx,mathtools, bbm}
\usepackage[colorlinks, plainpages]{hyperref}

\def\k{\mathbbm{k}}

\def \oR{{\overline {R}}}
\def \wR{{\widetilde {R}}}

\def \x{\langle x \rangle}

\def\fm{\frak{m}} 
\def\fp{\frak{p}}
\def\fq{\frak{q}}
 
\def\fb{\frak{b}}
\def\fc{\frak{c}}

\newcommand{\N}{{\mathbb N}}
\newcommand{\R}{{\mathbb R}}

\newcommand{\Z}{{\mathbb Z}}

\newcommand{\C}{{\mathbb C}}
\newcommand{\Q}{{\mathbb Q}}

\newcommand{\F}{{\mathbb F}}

\newcommand{\ka}{{\mathcal A}}

\newcommand{\kc}{{\mathcal C}}

\newcommand{\kl}{{\mathcal L}}

\newcommand{\ko}{{\mathcal O}}

\DeclareMathOperator{\Spec}{Spec}

\DeclareMathOperator{\Quot}{Quot}
\DeclareMathOperator{\Supp}{Supp}
\DeclareMathOperator{\Ann}{Ann}

\newtheorem{theorem}{Theorem}
\newtheorem{lemma}[theorem]{Lemma}
\newtheorem{corollary}[theorem]{Corollary}

\newtheorem{proposition}[theorem]{Proposition}
\newtheorem{definition}[theorem]{Definition}

\newtheorem{example}[theorem]{Example}
\newtheorem{remark}[theorem]{Remark}

\begin{document}

\title {On Delta for parameterized Curve Singularities}
\author{Gert-Martin Greuel and Gerhard Pfister}

\maketitle
\bigskip
\begin{abstract}
We consider families of parameterizations of reduced curve singularities
over a Noetherian base scheme and prove that the delta invariant is semicontinuous.
In our setting, each curve singularity in the family is the image of a parameterization and not the fiber of a morphism.  
The problem came up in connection with the right-left classification of  parameterizations of  curve singularities defined over a  field of positive characteristic.
We prove a bound for right-left determinacy of a parameterization in terms of delta and the semicontinuity theorem provides a simultaneous bound for the
determinacy in a family.
The fact that the base space can be an arbitrary Noetherian scheme causes some difficulties but is (not only) of interest for computational purposes.\\

\end{abstract}
\bigskip


\section*{Introduction}\label{intro}

Let $R$ be the local ring of a reduced curve singularity and $\oR$ its normalization, i.e., the integral closure in the total quotient ring of $R$. 
If $\k$ is a field and $\k \to R$ a ring map, then $\delta_\k(R):= \dim_\k \overline R/R$ is called the delta invariant of $R$ (w.r.t. $\k$),  the most important numerical invariant of a reduced curve singularity.
In the present paper we consider families 
of parameterizations of  reduced curve singularities, defined over an arbitrary field, and prove that the delta invariant behaves upper semicontinuous.

A parameterization of a reduced algebroid curve singularity is given 
(in the irreducible case) by a non-zero morphism
$$\varphi:  \k [[x_1,...,x_n]] \to \k [[t]],  \ \varphi(x_i) = \varphi^i(t), \ i=1,...,n,$$ 
where $\k$ is an arbitrary field. The image $\k[[\varphi^1(t),...,\varphi^n(t)]]
\subset \k [[t]]$ is a reduced 1-dimensional complete local ring, which we call a reduced curve singularity. 
A family of parameterization over a Noetherian ring $A$ is a morphism 
of $A$-algebras
$$\varphi_A: A[[x_1,...,x_n]] \to A[[t]], \ \varphi_A(x_i) = \varphi_A^i(t), i=1,...,n.$$
For a prime ideal 
$\fp \subset A$ with $k(\fp) = \Quot (A/\fp)$ the residue field of $\fp$, let 
$$\varphi_\fp:  k(\fp) [[x_1,...,x_n]] \to k(\fp) [[t]], \ \varphi_\fp(x_i) = \varphi_\fp^i(t), i=1,...,n,$$ 
be the induced map. If $\varphi_\fp \neq 0$, then the image of the parameterization $\varphi_\fp$ is
$k(\fp)[[\varphi_\fp^1(t),...,\varphi_\fp^n(t)]]  \cong R_\fp := k(\fp) [[x_1,...,x_n]]/Ker (\varphi_\fp)$, a reduced curve singularity, and we prove in  Theorem \ref{thm.scont} (also for not necessarily irreducible curve singularities) that the delta invariant is semicontinuous, i.e.,
$$\delta_{k(\fp)}(R_\fp) \geq \delta_{k(\fq)}(R_\fq) $$
for $\fq$ in some open neighbourhood of $\fp$ in $\Spec A$. 

For the proof we use some of the semicontinuity results from \cite{GP20}, together with the fact that the parameterization $\varphi_\fq$ is already determined (up to right-left equivalence) by its terms of 
order $\leq 4\delta_{k(\fq)}(R_\fq)-2$, which we prove 
 in Corollary \ref{cor.deter}. The semicontinuity of delta implies that this bound holds also for all parameterizations (of a given family) in a neighbourhood of $\fp$.
\medskip

 The delta invariant of a reduced curve singularity has been and still is a continuous subject of research (e.g. \cite{Hi65}, \cite{Te78}, \cite{Ca80}, \cite{Ca05}, \cite{CL06},  \cite{Ng16}, \cite{CLMN19}, \cite{IIL20}, to name a few). In all these papers the curve singularities appear as fibers $X_s$ of a flat morphism $X \to S$, and in this situation it is well known that $ \delta (X_s)$ is semicontinuous, assuming that the fibers are reduced (cf. \cite{Te78} for $S=\C$ and \cite{CL06} or \cite{GLS07} for $S$ normal).
The assumption, that the fibers are reduced is very restrictive, and it implies e.g. for $S=\C$ that $X$ is a Cohen-Macaulay surface singularity. 
\medskip

In our situation the curve singularities $R_\fq$ are for each $\fq$ the image of a parameterization and thus reduced as a subring of a reduced ring, but they are in general not the fiber of some morphism (see Remark \ref{rm.fiber} and Example \ref{ex.2}). We do not make any assumption on flatness and the base space of the family can be an arbitrary Noetherian scheme (e.g. $\Z$). 
This fact  and that we allow non closed points has interesting computational consequences, see  Remark \ref{rm.comp}. 
At the end of the paper we state an analogous result in the  analytic case, for which the proofs are much easier, see
Proposition \ref{prop.scont-an} and Remark \ref{rm.scont-an}.


\section{Parameterized curve singularities}\label{sec.2}
Throughout the paper we assume all rings to be associative, commutative and with 1, that ring maps map 1 to 1 and that maps of local rings respect maximal ideals. $\k$ denotes an arbitrary field and 
 $A$ and $R$ denote Noetherian rings. 

 Let $R$ be reduced. Then $Quot(R)$, the {\em total quotient ring of  $R$}, is a direct product of the fields $Quot(R/P_i) $, where
$P_1,...,P_r$ are the minimal primes of $R$ (\cite [p.183 proof of Th. 23.8] {Mat86}).   $\oR$ denotes the {\em integral closure} of $R$ in $Quot(R)$, it is the direct product of the rings $\overline {R/P_1}, ..., \overline {R/P_r}$,
and the natural inclusion  $R \hookrightarrow \overline R$ is called the {\em normalization of $R$}. 
  The annihilator of the $R$-module $\oR/R$,  
 $\kc :=\Ann_R(\oR/R) \subset R,$ 
 is called the {\em conductor ideal} of $R \subset \oR$. 

\begin{remark}\label{rm.nnor}
{\em
The non-normal locus $\Supp_R(\oR/R)$ may not be closed in $\Spec R$. However, if $\oR/R$ is (module-) finite over $R$, then 
$\Supp_R(\oR/R)$  coincides with $V(\kc) = \{\fp \in \Spec R \, | \, \fp \supset \kc \}$ and is therefore closed in $\Spec R$. }
\end{remark}

 The following characterization of the normalization is useful.

\begin{proposition}\label{prop.charnor}
A morphism of reduced Noetherian rings $\nu: R \to \wR$ is the normalization map
$\Leftrightarrow$ $\wR$ is normal and  $\nu$ is integral and birational.\footnote { A morphism $\varphi: A \to B$ between Noetherian rings is {\em birational} if $\varphi$ induces a bijection between the minimal primes of $B$ and $A$ and if for every minmal prime $\fp$ of $B$ the induced morphism of local rings $A_{\varphi^{-1}(\fp)} \to B_\fp$ is an isomorphism.}
\end{proposition}
\begin{proof}
The direction $\Rightarrow$ follows from \cite[Lemma 28.52.5(2) and Lemma 28.52.7]{Stack} and the converse from  \cite[Lemma 28.52.5(3), Lemma 28.52.7, and Lemma 28.5.8]{Stack}.
\end{proof}

 \begin{definition} \label{def.delta}
 Let $(R,\fm)$ be a reduced local  ring with normalization $\oR$, $\k$ a field, and $\k \to R$ a ring map.
\begin{itemize} 
\item [(i)] The {\em  delta invariant} of $R$ (w.r.t. $\k$) is  defined as
$$\delta_\k(R):= \dim_\k \overline R/R .$$
\item [(ii)] The {\em (multiplicity of the) conductor} of $R$ (w.r.t. $\k$) is defined as
$$c_\k (R):=  \dim_\k \oR/\kc.$$
\end{itemize}
\end{definition}

\begin{remark} {\em
If $(R,\fm)$ is a one-dimensional reduced Nagata ring and $\hat R$ its $\fm$-adic completion, then  $\hat\oR$ is the normalization of  $\hat R$ and
$\delta_\k(R) = \delta_\k(\hat R)$. This is a consequence of Lemma 32.40.2 in \cite{Stack}.  }
\end{remark}

 We want to study the behavior of $\delta_\k(R)$ in a family of parameterizations. We start with a parameterizations of a reduced and irreducible curve singularity, the {\em uni-branch case}.
Let $\k$ be a field and $x_1,...,x_n$, $t$  independent variables. Consider the power series rings
\begin{align*}
P:= & \ \k[[x]]:=\k[[x_1,...,x_n]], \text { resp. } \wR :=  \k[[t]],
\end{align*}
with maximal ideals $ \x:=  \langle x_1,...,x_n\rangle P$ resp. $\widetilde\fm := t\wR$. 

\begin{definition}\label{def.param1}
Let
\begin{align*}
\varphi :  P \to \wR,  \  \varphi(x_i) =: \varphi^i(t)  \in t\k[[t]], \ i=1,...,n,
\end{align*}
be a  morphism of local $\k$-algebras
such that $\varphi(x_i) \neq 0$ for at least one $i$. We set
$$R:=P/Ker(\varphi), $$
which is isomorphic to $\varphi(P)= \k[[\varphi(x_1),\cdots, \varphi(x_n)]] \subset \k[[t]] =\wR$.
\begin{enumerate}
\item We call  $\varphi :P \to \wR$ 
a {\em parameterization of the uni-branch curve singularity}
 $ \varphi(P)$ (or of $R$).
\item $\varphi$  is called {\em primitive} if for any 
morphism $\tilde \varphi: P \to \k[[t]]$, satisfying $\varphi^i(t) = \tilde \varphi^i(\tau (t))$ for some power series $\tau$ and $i=1,...,n$,  the order of $\tau$ is 1. 
\end{enumerate}
\end{definition}

\begin{lemma} \label{lem.prim1} 
Let $ \varphi : P \to \wR$ be a parameterization of the uni-branch curve singularity $\varphi(P)$ as in Definition \ref{def.param1}. Then $\varphi$ is a finite morphism and $R \cong \varphi(P)$
is  a Noetherian, reduced and irreducible complete local ring of dimension 1 with maximal ideal 
$\x R \cong \langle \varphi(x)\rangle \k[[t]]$. 
Moreover,
the following are equivalent:
\begin {enumerate}
\item [(i)] $\varphi$  is a primitive parameterization,
\item [(ii)]  $\delta_\varphi := \dim_\k \wR/\varphi(P) <\infty$,
\item [(iii)]  $\varphi(P) \hookrightarrow \wR$ is the normalization of $\varphi(P)$.
\end {enumerate}
\end{lemma}

\begin{proof}
Since $\varphi^i\neq 0$ for some $i$, the inclusion  
$\varphi(P) \subset \k[[t]]$ 
is quasi-finite and hence finite by the Weierstrass Finite Theorem \cite[Theorem 1.10]{GLS07}. By \cite[Corollary 3.3.3]{GP08}  
$\dim \varphi(P)  = \dim \k[[t]]=1$.
Moreover, $\Spec \k[[t]] \to \Spec \varphi(P) $ is surjective (\cite[Lemma 10.35.17]{Stack}) and thus $\varphi(P)$ and hence 
$R$ is the local ring of an irreducible, reduced curve singularity.  

It follows that the normalization  $\oR$ (integral closure in $\Quot(R)$)  of $R$ is of the form $\k[[\tau]]$ for some variable 
$\tau$ and $\dim_\k \oR/R < \infty$ (\cite[32.40.2]{Stack} together with Cohen's structure theorem).
Since $\k[[t]]$ is integrally closed 
$k[[\tau]] \hookrightarrow \k[[t]]$, $\tau \mapsto \tau(t)$, and
we have inclusions 
$$\varphi(P)= \k[[\varphi^1(t),\cdots, \varphi^n(t)]] \cong R \subset \oR \cong k[[\tau]] \hookrightarrow  \k[[t]] =\wR.$$
If $\varphi$ is primitive, then $ ord (\tau(t)) =1$, $\oR \cong \wR$, 
and $\wR$ is the normalization of $\varphi(P)$.

If $\varphi$ is not primitive there exists a
morphism $\tilde \varphi: P \to \k[[t]]$, satisfying $\varphi^i(t) = \tilde \varphi^i(\tau (t))$, $i=1,...,n$,  for some power series $\tau$ with order of $\tau$  $\geq 2$. We have 
$\k[[\varphi^1(t),...,\varphi^n(t)]] \subset \k[[\tau (t)]] \subset \k[[t]]$
and $\dim_\k  \k[[t]]/ \k[[\tau (t)]] =\infty$. Hence  $\dim_\k \wR/\varphi(P) 
=\infty$ and $\wR$ is not the normalization of $R$.
\end{proof}

Note that the parameterizations
$\varphi_1=(t^2,t^3)$ and $\varphi_2=(t^4,t^6)$ define the same curve singularity $R=\k[[x,y]]/\langle x^3-y^2\rangle$,
but with different embeddings $\varphi_1(P)$ and $\varphi_2(P)$ in $\k[[t]]$. The first parameterization is primitive, the second not. \\

We consider now the {\em multi-branch case}, i.e., the parameterization of a curve singularity with  several branches.
Let  $t_1,...,t_r$ be independent variables, set
\begin{align*}
\wR :=  \ \k[[t_1]]\oplus ... \oplus \k[[t_r]],
\end{align*}
and let  $\pi_j : \wR \to  \wR_j:= \k[[t_j]]$ denote the projection. 
$\wR$ is a  complete,  semilocal, principal ideal ring with Jacobson radical 
$$\widetilde \fm := \widetilde \fm_1\oplus...\oplus \widetilde \fm_r, \
\widetilde\fm_j := t_j\k[[t_j]],$$
and $\widetilde\fm_j$ the maximal ideal of the local ring $\wR_j$.

\begin{definition}\label{def.param}
Let
\begin{align*}
\begin{split}
&\varphi :  P =\k[[x_1,...,x_n]] \to \wR,  \\  
&\varphi(x_i) =: 
(\varphi_1^i(t_1),..., \varphi_r^i(t_r))\in  t_1\k[[t_1]] \oplus ... \oplus t_r\k[[t_r]],
\end{split}
\end{align*}
be a morphism of $\k$-algebras. For  $j=1,...,r$ let
$$\varphi_j  := \pi_j \circ \varphi : P \to \wR_j=\k[[t_j]] $$ 
be the composition. We set  
 \begin{align*}
 R&:= P/Ker(\varphi) \cong \varphi(P), \\  
 R_j &:= P/Ker(\varphi_j) \cong \varphi_j(P).
 \end{align*}
\begin{enumerate}
\item
We call  $\varphi =(\varphi_1,...,\varphi_r)$ 
a {\em parameterization of the curve singularity} $ \varphi(P)$  (or of $R$) with $r$ branches if $\varphi_j$ is a parameterization of
$\varphi_j(P)$ as in Definition (\ref{def.param1}) and if  $\varphi_j(P) \neq \varphi_{j'}(P)$  for $j \neq j'$. The morphism of local rings
 $\varphi_j$, $ j=1,...,r$, is called a {\em parameterization of the branch} $ \varphi_j(P)$,
 (or of $R_j$).  
\item The parameterization $\varphi$ is called {\em primitive} if $\varphi_j$ is primitive for each $j$.
\end{enumerate}
\end{definition}

The rings 
\begin{align*}
& \varphi(P)= \k[[\varphi(x_1),\cdots, \varphi(x_n)]] \subset \k[[t_1]]\oplus ... \oplus \k[[t_r]] \text{ and } \\
&\varphi_j(P) = \k[[\varphi_j^1(t_j),\cdots, \varphi_j^n(t_j)]] \subset \k[[t_j]]
\end{align*}
are Noetherian, reduced, complete local rings with maximal ideals $\fm:= \varphi(\x)$ and 
 $\fm_j:= \varphi_j(\x)$ respectively.
 It follows from Lemma \ref{lem.prim1} that $\varphi_j(P)$ is the local ring of  a reduced, irreducible curve singularity. Hence $\varphi(P)$ is the local ring of a reduced curve singularity with $r$ {\em branches}. 
 
\bigskip
 
Defining the {\em conductor ideal of $\varphi$} as  
$$\kc_\varphi := \Ann_{\varphi(P)}(\wR/\varphi(P)),$$ 
we get the following result for multi-branch curve singularities.

 \begin{lemma} \label{lem.prim2}
Any  parameterization $ \varphi : P \to \wR$ of  $\varphi(P)$ 
as in Definition \ref{def.param} is a finite morphism and $R \cong \varphi(P)$
is  a Noetherian, reduced and complete local ring of dimension 1. 
Moreover, the following are equivalent:
\begin {enumerate}
\item [(i)] $\varphi$  is a primitive parameterization,
\item [(ii)]  $\delta_\varphi := \dim_\k \wR/\varphi(P) <\infty$ {\em (delta-invariant of $\varphi$)},
\item [(iii)]  $c_\varphi := \dim_\k \wR/\kc_\varphi <\infty$ {\em (conductor of  $\varphi$)}, 
\item [(iv)] $\varphi(P) \hookrightarrow \wR$ is birational,
\item [(v)]  $\varphi(P) \hookrightarrow \wR$ is the normalization of $\varphi(P)$.
\end {enumerate}
 \noindent If any of these conditions is fulfilled, then 
 $\delta_\k(R) = \delta_\varphi$ and $c_\k(R) = c_\varphi$. Moreover, 
 $\delta_\k(R) \leq c_\k(R) \leq 2\delta_\k(R)$ and $c_\k(R) = 2\delta_\k(R)$ iff $R$ is Gorenstein.
 \end{lemma}

\begin{proof} $\wR_j$ is a finite $R_j$-module by Lemma \ref{lem.prim1}  and since
$R_1 \oplus ... \oplus R_n$ is finite over $R$, 
$\wR = \wR_1\oplus...\oplus \wR_n$ is finite over $R$. 

With $I_j := Ker(\varphi_j)$ we have  $\varphi_j(P) \cong P/I_j$,  $j=1,...,r$. 
Consider, for $r=2$, the inclusions 
$$ \varphi(P) \cong P/I_1\cap I_2 \hookrightarrow P/I_1\oplus P/I_2 \hookrightarrow \wR_1 \oplus \wR_2 = \wR$$
and the exact sequence 
$$0 \to P/I_1\cap I_2 \to P/I_1\oplus P/I_2 \to P/(I_1+I_2) \to 0.$$
It follows 
$\delta_\varphi = \delta_{\varphi_1} + \delta_{\varphi_2} + \dim_\k P/(I_1+I_2)$,
with $\dim_\k P/(I_1+I_2) < \infty$ since $P/(I_1+I_2)$ is concentrated on $\fm$ (by definition of a parameterization). By induction on $r\geq 2$ we see that 
$\delta_\varphi < \infty$ iff $\delta_{\varphi _j} < \infty$ for all $j=1,...,r$.

The equivalence of (i), (ii), (v) is now a consequence of Lemma 
\ref{lem.prim1}. The equivalence of (iv) and (v) follows from Proposition 
\ref{prop.charnor}.

The equivalence of (ii) and (iii) can be seen as follows:  
The exact sequence
 $$0 \to  \varphi(P)/\kc_\varphi  \to  \wR/\kc_\varphi \to \wR/ \varphi(P) \to 0 $$
implies $c_\varphi = \delta_\varphi + \dim_\k \varphi(P)/\kc_\varphi$ and hence $\delta_\varphi \leq c_\varphi$.
We claim $\delta_\varphi  <\infty$ implies   $c_\varphi < \infty$. 
    In fact, $\kc_\varphi$  is an $\wR$-ideal in $ \varphi(P)$ and 
since $\wR$ is a principal ideal ring, we get 
$\kc_\varphi = t^{c_\varphi}  \wR$  if $\kc_\varphi\neq 0$.    
We have  $\varphi(x_i)\neq0$ for some $i$ and hence
  $\varphi(x_i)^k (\wR/\varphi(P))=0$ for some $k$ ($\delta_\varphi  <\infty$)
  by Nakayama's lemma since $\wR/\varphi(P)$ is finite over
  $\varphi(P)$. It follows that  $\kc_\varphi \neq 0$ and hence $c_\varphi <\infty$. 
  
The relations between  $\delta_\k(R)$ and  $c_\k(R)$ are well known. In fact,
from the exact sequence  
 $0 \to  R/\kc  \to  \oR/\kc \to \oR/ R \to 0 $
 we get 
 $c_\k(R) = \delta_\k(R) + \dim_\k R/\kc.$ 
By \cite[Korollar 3.7]{HK71} we have $2 \dim_\k R/\kc \leq c_\k(R)$, with equality iff $R$ is Gorenstein. 
Hence $ \dim_\k R/\kc  \leq \delta_\k(R)$ with equality iff $R$ is Gorenstein.  
This implies the claim.
\end{proof}

\begin{remark} {\em
1. By Lemma \ref{lem.prim2} $\delta_\varphi $ and $c_\varphi$ depend only on $R$ (and not on the embedding $\varphi$) iff $\varphi$ is primitive. 

2. If $\varphi$ is a uni-brach parameterization, we can define 
the {\em semigroup (of values) of $\varphi$} by setting
$ \Gamma_\varphi := \{v(g) \ | \ g \in \varphi(P) \} \subset \N$ with $v(g)$ the order of 
$g \in \k[[t]]$. 
 It is easy to see that  the set $ \N \setminus \Gamma_\varphi $ is finite iff the greatest common divisor of the integers in $\Gamma_\varphi $ (equivalently, of any set of semigroup generators of $\Gamma_\varphi $) is 1.  The cardinality  $\sharp( \N \setminus \Gamma_\varphi)$  ("number of gaps") is called the {\em delta invariant} of $\Gamma_\varphi$ and denoted by   $\delta(\Gamma_\varphi)$. 
 The smallest integer $c$ in $\Gamma_\varphi$ such that $c+\N \subset \Gamma_\varphi$
is called the {\em conductor} of $\Gamma_\varphi$ and denoted by $c(\Gamma_\varphi)$ if it exists, otherwise we set 
$c(\Gamma_\varphi):=\infty$. 

We have
$\delta_\varphi = \delta(\Gamma_\varphi)$ and  $c_\varphi = c(\Gamma_\varphi)$ for any parameterization $\varphi$ and  
(by  Lemma \ref{lem.prim2}) $\delta_\k(R) = \delta(\Gamma_\varphi)$ and 
 $c_\k(R) = c(\Gamma_\varphi)$ iff $\varphi$ is primitive.
}
 \end{remark}
 \medskip
 
 At the end of this section we show that a primitive parameterization $\varphi= (\varphi_1,...,\varphi_r)$ is isomorphic to a polynomial parameterization of degree $< 2c_\varphi$, i.e.  the $\varphi_j \in \k[[t_j]]$ are determined by their terms of degree $< 2c_\varphi$.
 
 More precisely,  with $Aut_{\k}(P)$ resp. $Aut_\k(\wR)$ the groups of
 $\k$-algebra automorphisms of $P$ resp. $\wR$, we make the following definition.
 
 \begin{definition}\label{def.deter}
\begin{enumerate}
\item
Let $\varphi,\psi:P\longrightarrow \tilde R$ be two parameterizations. We say that  $\varphi$ and $\psi$ are {\em right-left} or {\em $\ka$-equivalent}
($\varphi \sim_{\mathcal A} \psi$) if there exist
$\sigma \in \text{Aut}_{\k}(P)$ and $\tau \in \text{Aut}_{\k}(\tilde R)$ such that $\tau\circ\varphi\circ\sigma^{-1}=\psi$. If $\tau = id$, we say that
$\varphi$ and $\psi$ are {\em left} or {\em $\kl$-equivalent} ($\varphi \sim_{\mathcal L} \psi$). 
\item
$\varphi$ is called $ k - \ka$-determined (resp. $ k - \kl$-determined) if $\varphi \equiv \psi \text{ mod } \widetilde \fm^{k+1}$ implies  $\varphi \sim_{\mathcal A} \psi$ (resp.$\varphi \sim_{\mathcal L} \psi$).
\end{enumerate}
\end{definition}

\begin{proposition} \label{prop.deter}
Let $\varphi:P\longrightarrow \tilde R$ be a parameterization with finite conductor $c_\varphi$.Then $\varphi$ is $(2c_\varphi-1) - \kl$-determined and hence $(2c_\varphi-1)-\ka$-determined.
\end{proposition}

\begin{proof} Since $\wR$ is a principal ideal ring, $\kc = f \wR$, with 
$f = (t_1^{c_1},...,t_r^{c_r}), c_i\geq 1,$ and $c_\varphi = c_1+...+c_r$. Set 
$c:= max\{c_1,...,c_r\}$ and we are going to show that $\varphi$ is $(2c-1) - \kl$-determined.

Since $c\geq c_i$ for all $i$ we get the $\widetilde \fm ^c \subset \kc \subset \widetilde \fm $
with $\widetilde \fm = (t_1,...,t_r)\wR$ the Jacobson radical of $\wR$.
Moreover, each $h\in \kc \subset \varphi(P)$ satisfies $h=H(\varphi(x_1),\ldots,\varphi(x_n))$ for a suitable $H\in \x \k[[x_1,\ldots,x_n]]$.

Assume now that $\varphi \equiv \psi \text{ mod } \widetilde \fm^{2c}$, 
i.e. $\varphi(x_i)=\psi(x_i)+h_i$ with 
$h_i\in \widetilde \fm^{2c} \subset f  \widetilde \fm^{c}$, i.e. $h_i=fg_i$
with $g_i\in  \fm^{c}$. We obtain that $f=F(\varphi(x_1),\ldots,\varphi(x_n))$ 
and $g_i=G_i(\varphi(x_1),\ldots,\varphi(x_n))$ for suitable $F, G_i \in \x \k[[x_1,\ldots,x_n]]$. Hence $h_i = H_i ((\varphi(x_1),\ldots,\varphi(x_n))$
with $H_i=FG_i \in \x^2 \k[[x_1,\ldots,x_n]]$.
We obtain that 
$$\varphi(x_i)-H_i(\varphi(x_1),\ldots,\varphi(x_n))=\psi(x_i).$$
Now let $\sigma:P\longrightarrow P$ be the automorphism
defined by $$\sigma(x_i)=x_i-H_i(x_1,\ldots,x_n)$$
then we obtain $\phi\circ\sigma=\psi$ and this implies $\varphi \sim_{\mathcal L} \psi$.
\end{proof}
\begin{corollary}\label{cor.deter}
Let $\varphi:P\longrightarrow \tilde R$ be a parameterization. If $\delta_{\varphi}<\infty$ then $\varphi$ is $(4\delta_{\varphi}-2)- \kl$-determined.
\end{corollary}
\begin{proof}
Since $2\delta_{\varphi}\geq c_{\varphi}$ by Lemma \ref{lem.prim2}  we obtain the result.
\end{proof}

\begin{remark} {\em
1. $\ka$-equivalence of $\varphi$ and $\psi$ implies that the $\k$-algebras
 $\varphi(P)$ and $\psi(P)$ are isomorphic (this is sometimes called {\em contact equivalence}). Of course, the multiplicity of the conductor and the delta invariant depend only on the contact equivalence class of a parameterization. 
Hironaka showed in \cite[Theorem B]{Hi65} that a reduced algebraic curve singularity $R$, defined over an algebraically closed field  $\k$, is $3\delta_\k(R)+1$ contact determined, which can be improved for plane curve singularities over the complex numbers to $2 \delta_{\mathbb C}(R) -r +2 $ (cf. e.g. \cite [Corollary I.2.24]{GLS07}).
 
 2. The proof of Proposition \ref{prop.deter} shows a bit more then claimed. First, $\varphi$ $2c-\kl$-determined with $c= max\{c_1,...,c_r\} \leq c_\varphi$. Second, the automorphism $\sigma:P\longrightarrow P$ is the identity on linear terms, hence $\varphi$ and $\psi$ are in the same orbit of the corresponding unipotent subgroup of $Aut_{\k}(P)$. 
 
 3. For an integer $c>1$ the parameterization  $\varphi = (t^c, t^{c+1}, ...,t^{2c-1})$ has $c_\varphi =c$ and $\delta_\varphi = c-1$. By Proposition \ref{prop.deter} $\varphi$ is $(2c-1)-\kl$-determined, but it is obviously not $2c-k$-determined for $k > 1$, showing that the bound is sharp in this example (contrary to the too optimistic bound $c$ that can be found in the literature,
 e.g in \cite[Proposition 2.1]{Ca05}). Parameterizations of plane curve singularities  are $(c_\varphi+1)-\ka$-determined in any characteristic as shown in \cite{Ng16}.
 
 4. The advantage of the bound $4\delta_\varphi-2$ in Corollary \ref{cor.deter} 
 compared to $2c_\varphi-1$ in Proposition \ref{prop.deter} is that the delta invariant is semicontinuous in a family (Theorem \ref{thm.scont})
while the multiplicity of the conductor is not semicontinuous (Example \ref{ex.1}). This is important for the construction of versal deformations  or for the classification  of parameterizations  but also for computational purposes (Remark \ref{rm.comp}).
}\end {remark}

 \section{Semicontinuity of delta for families of parameterizations} \label{sec.3}

We analyze the semicontinuity of delta in a family of parameterizations of reduced curve singularities, which we define now.

\begin{definition} \label{def.fampar}
Let $A$ be a Noetherian ring.
\begin{enumerate}
\item  Consider a  morphism of $A$-algebras
$$ \quad \quad  \varphi_A: P_A:= A[[x]]=A[[x_1,...,x_n]] \to \wR_A := A[[t_1]]\oplus ... \oplus A[[t_r]]$$
and denote by 
$$\varphi_{A,j}: P_A \to \wR_{A,j}:=A[[t_j]] $$
the composition of $ \varphi_A$ with the projection $\wR_A \to A[[t_j]]$.

\noindent We assume that $\varphi_{A,j}(x_i) \in t_jA[[t_j]] $ for $i=1,...,n$, $j=1,...,r$
and that
$\varphi_{A,j}(P) \neq \varphi_{A,j'}(P)$ for $j\neq j'$.
\medskip

\item
The $A$-algebras
\begin{align*}
\quad \quad \varphi_A(P_A)&= A[[\varphi_A(x_1),\cdots, \varphi_A(x_n)]] \subset A[[t_1]]\oplus ... \oplus A[[t_r]] \text{ and } \\
\quad \quad \varphi_{A,j}(P_A) &= A[[\varphi_{A,j}^1(t_j),\cdots, \varphi_{A,j}^n(t_j)]] \subset A[[t_j]]
\end{align*}
are Noetherian and we set  
 \begin{align*}
 R_A &:= P_A/Ker(\varphi_A) \cong  \varphi_A(P_A), \\
R_{A,j} &:= P_A/Ker(\varphi_{A,j}) \cong  \varphi_{A,j}(P_A).
 \end{align*}

\item For a prime ideal  $\fp$ of $A$ let $k(\fp) = A_\fp/\fp A_\fp =\Quot(A/\fp)$ be the residue field of $\fp$. We define
\begin{align*}
P_\fp & := k(\fp)[[x_1,...,x_n]], \\
 \wR_\fp & := k(\fp)[[t_1]]\oplus ... \oplus k(\fp)[[t_r]],\\
 \wR_{\fp,j} & := k(\fp)[[t_j]].
\end{align*} 
Then $\varphi_A$ induces morphisms
\begin{align*}
\varphi_\fp: P_\fp & \to \wR_\fp,\\
\varphi_{\fp,j}:  P_\fp & \to \wR_{\fp,j}, 
\end{align*} 
with $\varphi_{\fp,j}$ the composition of $\varphi_\fp$ with the projection to $\wR_{\fp,j}$.
We set
\begin{align*}
R_\fp & := P_\fp/Ker(\varphi_\fp)  \cong \varphi_\fp(P_\fp),    \\ 
R_{\fp,j}& := P_\fp/Ker(\varphi_{\fp,j}) \cong \varphi_{\fp,j} (P_\fp).
\end{align*} 
\end{enumerate}
and call  $\varphi_A: P_A \to \wR_A$ or $\{\varphi_\fp: P_\fp \to \wR_\fp\}_{\fp\in \Spec A}$  a {\em family of parameterizations} over $A$. 
\end{definition}

\begin{remark}\label{rm.cond} {\em
Fix $\fp \in \Spec A$.  Then $\varphi_{\fp}$   is   a parameterization of the reduced curve singularity 
$\varphi_{\fp} (P_\fp)$ with $r$ branches $\varphi_{\fp,1} (P_\fp),...,\varphi_{\fp,r} (P_\fp)$ in the sense of Definition \ref{def.param} if and only if the following condition holds (cf. Lemma \ref{lem.prim2}):
\noindent
\begin{align} \tag {*} \label{cond}
\begin{split}
& \bullet  \text {for each } j  \text{ there is at least one } i \text{ s.t. } \varphi_{\fp,j}(x_i)\neq 0, \\
& \bullet   \text{for } j\neq j' \text{ we have } \varphi_{\fp,j}(P_\fp) \neq \varphi_{\fp,j'}(P_\fp). 
\end{split}
\end{align}
 If condition (\ref{cond}) holds for $\fp$ then it holds  for $\fq$ in some open neighbourhood $U$ of $\fp$ in $\Spec A$. In particular, the number of branches
 of  $\varphi_{\fq}(P_\fq)$ is constant for $\fq \in U$.
 }
\end{remark}
\bigskip

Let us recall from \cite{GP20} the completed tensor product and the completed fiber, the main tool for our approach.

Let $M$ be any $P_A$-module and $\fp \in \Spec A$. Considering $M$ as an $A$-module, we set
$$M(\fp):= M \otimes_{A} k(\fp)$$
and call it the {\em fiber of $M$ over $\fp$}. The {\em completed fiber} 
of  $M$ over $\fp$ is defined as 
$$  \hat M(\fp) :=  M  \hat\otimes_A k(\fp) \cong M(\fp)^\wedge,$$
where $N^\wedge$ denotes the $\x$-adic completion of a $P_A$-module $N$ and 
$ \hat\otimes$ denotes the {\em completed tensor product}, see \cite[Definitions 2 and 11, and Proposition 3]{GP20}.
In general, if $B$ is any $A$-algebra, then 
$$A[[x]] \hat\otimes_A B = (A[[x]] \otimes_A B)^\wedge = B[[x]]$$ 
(\cite[Proposition 3]{GP20}). If $M$ is a finitely presented $A[[x]]$-module, then 
$M \hat\otimes_A B$ is a finitely presented  $B[[x]]$-module (\cite[Corollary 6]{GP20}). 

Similarly we have $A[[t]] \hat\otimes_A B = (A[[t]] \otimes_A B)^\wedge = B[[t]]$, where here $\wedge$ denotes the $\langle t \rangle$-adic completion, and  $\hat\otimes_A$ commutes with direct sums.

\begin{lemma} \label{lem.fiber}
With the notations from above and from  Definition \ref{def.fampar}, 
let $\wR_A$ be finite over $P_A$.
For $\fp \in \Spec A$ we have
\begin{enumerate}
\item $P_\fp = P_A \hat\otimes_A k(\fp)$,
\item $\wR_\fp =\wR_A  \hat\otimes_A k(\fp)$,
\item $\varphi_\fp = \varphi_A\hat\otimes_A k(\fp): P_\fp  \to \wR_\fp,$ \\
factoring as
$\varphi_\fp :  P_\fp \twoheadrightarrow R_\fp \cong \varphi_\fp(P_\fp) \hookrightarrow  \wR_\fp$,
\item $ (\wR_A /R_A )\hat\otimes_A k(\fp)=\wR_\fp/R_\fp.$
\end{enumerate}
\end{lemma}

\begin{proof}
(1), (2) and (3) follow from \cite[Proposition 3]{GP20}. If we tensor the exact sequence 
$$ P_A \xrightarrow{\varphi_A} \wR_A \to  \wR_A/\varphi_A(P_A) \cong \wR_A /R_A \to 0$$
with $\hat\otimes_A k(\fp)$ we get
$$ (\wR_A /R_A )\hat\otimes_A k(\fp)
= (\wR_A \hat\otimes_A k(\fp)/Im(\varphi_A\hat\otimes_A k(\fp))=\wR_\fp/R_\fp,$$
since $\hat\otimes_A$ is right exact for finitely generated $P_A$-modules (\cite[Corollary 4]{GP20}). 
\end{proof}

\begin{remark} \label{rm.fiber}{\em
In general
 $\varphi_\fp$ is not the restriction of  $\varphi_A $ to the fiber over $\fp$
 for an arbitrary prime
$\fp \in \Spec A$, i.e., $\varphi_\fp \neq \varphi_A\otimes_A k(\fp)$
(cf. \cite[Remark 13]{GP20}), except for maximal ideals.
Namely, if  $\fp$ is a maximal ideal of $A$, then $k(\fp) = A/\fp$ and it follows $P_\fp =  k(\fp)  [[x]] = P_A\otimes_A k(\fp)$ and
$\wR_\fp =\wR_A  \otimes_A k(\fp)$ and hence $\varphi_\fp = \varphi_A\otimes_A k(\fp)$. 

However  
$R_\fp \neq R_A\otimes_A k(\fp)$ and $R_\fp \neq R_A\hat\otimes_A k(\fp)$ in general, even if $\fp$ is maximal (cf. Example \ref{ex.2}). 
The relation between $R_\fp$ and   $R_A\hat\otimes_A k(\fp)$ for arbitrary $\fp \in \Spec A$ is as follows. 
Tensoring
$$0 \to Ker(\varphi_A) \to P_A \to R_A \to 0$$ with  $\hat\otimes_A k(\fp)$ we have 
by the right-exactness of $\hat\otimes_A$ (\cite[Corollary 4]{GP20})
 $$R_A\hat\otimes_A k(\fp)=P_\fp/Ker(\varphi_A)P_\fp$$ 
  with $Ker(\varphi_A)P_\fp$ 
the image of $Ker(\varphi_A)\hat\otimes_A k(\fp)$ in $P_\fp$. Moreover, 
tensoring
$$ \varphi_A: P_A \twoheadrightarrow R_A \cong \varphi_A(P_A) \subset \wR_A$$
with  $\hat\otimes_A k(\fp)$ we get 
$$ \varphi_\fp: P_\fp \twoheadrightarrow R_A\hat\otimes_A k(\fp)  \to \wR_\fp.$$
Since $R_\fp =  P_\fp/Ker(\varphi_\fp) \cong \varphi_\fp (P_\fp)$ we have an induced maps 
 $$\R_A\hat\otimes_A k(\fp) \twoheadrightarrow R_\fp \hookrightarrow \wR_\fp.$$ 
In contrast to $R_\fp$ the ring $R_A\hat\otimes_A k(\fp)$ may not be reduced and   
$\R_A\hat\otimes_A k(\fp) \to R_\fp$ is in general not injective as  Example \ref{ex.2} shows.

}
\end{remark}

The main result of this paper is the following theorem, confirming  for
a family of parameterizations of reduced curve singularities the {\em (upper) semicontinuity} of the function 
$$\fp \mapsto \delta_{k(\fp)} (R_\fp) = \dim_{k(\fp)} (\oR_\fp/R_\fp)$$
on $\Spec A$. 
This means that for $\fp \in A$, 
 there exists an open neighbouhood $U$ of $\fp$ in $\Spec A$ such that  
 $\delta_{k(\fq)} (R_\fq) \leq  \delta_{k(\fp)} (R_\fp)$ for each $\fq \in U$.

\begin{theorem}\label{thm.scont}
Let $\varphi_A: P_A \to \wR_A$ be a family of parameterizations as in Definition \ref{def.fampar} and fix  $\fp \in \Spec A$. Assume that 
$\varphi_{\fp}:  P_\fp \to  \wR_\fp$  is  a parameterization of  the reduced curve singularity $R_\fp \cong \varphi_{\fp} (P_\fp)$ (satisfying condition (\ref{cond}) of Remark \ref{rm.cond}) 
with  $\dim_{k(\fp)} (\wR_\fp/R_\fp) < \infty$.
Then there exists an open neighbourhood $U \subset \Spec A$ of $\fp$  such that
\begin{enumerate}
\item  $R_\fq  \hookrightarrow \wR_\fq$ is the normalization of $R_\fq$ 
  for each $\fq \in U$ and 
 \item $\fq \mapsto \delta_{k(\fq)} (R_\fq)$  is upper semicontinuous on $U$.
\end{enumerate}
 \end{theorem}

\begin{proof}
By Lemma  \ref{lem.prim2}  we have
$\dim_{k(\fp)} (\wR_\fp/R_\fp) =\delta_{k(\fp)} (R_\fp)$ and  
with $N:= 4\delta_{k(\fp)} (R_\fp)-2$ we get from 
Corollary  \ref{cor.deter} that $\varphi_\fp$ is $N-\kl$-determined. 
We define 
$$ \bar \varphi_A: P_A \to \wR_A, \ x_i \mapsto \bar \varphi_A(x_i)$$
with
$\bar \varphi_A (x_i) =  \varphi_A (x_i) \text { mod } t_1^{N+1}A[[t_1]]\oplus ... \oplus t_r^{N+1}A[[t_r]]$,
identified with its terms up to order
$N$ in $\wR_A = A[[t_1]]\oplus ... \oplus A[[t_r]]$.

For $\fq \in \Spec A$ we have $\bar \varphi_A \hat \otimes_A k(\fq) = \bar \varphi_\fq: P_\fq \to \wR_\fq$, with $\bar \varphi_\fq (x_i) =  \varphi_\fq (x_i) \text { mod } \widetilde \fm^{N+1}$ and $ \bar \varphi_\fp \sim_\kl \varphi_\fp$
by Corollary  \ref{cor.deter}.
$ \bar \varphi_A$ is a polynomial map and Proposition 
\ref{prop.semic} below implies  $\delta_{\bar\varphi_\fq} = \dim_{k(\fq)} (\wR_\fq / \bar\varphi_\fq(P_\fq) \leq \delta_{\bar\varphi_\fp}$ for $\fq$  in some neighbourhood $U$ of $\fp$. Hence $\bar\varphi_\fq$ is 
$N-\kl$-determined for all $\fq \in U$. Since the delta invariant is invariant under $\kl$-equivalence, we get
$$  \delta_{k(\fp)}(R_\fp) = \delta_{\varphi_\fp} =  \delta_{\bar\varphi_\fp} \geq \delta_{\bar\varphi_\fq}
=  \delta_{\varphi_\fq} = \delta_{k(\fq)} (R_\fq)$$
for $\fq \in U$. Together with Lemma  \ref{lem.prim2} this proves the theorem.
\end{proof}

\begin{proposition} \label{prop.semic}
With the notations and assumptions of Theorem \ref{thm.scont}, assume moreover that  $\varphi_A: P_A \to \wR_A$ is defined by algebraic power series, i.e., 
$\varphi_{A,j}(x_i) \in A\langle t_j\rangle$ (e.g.  $\in A[t_j]$) for $j=1,...,r, \ i=1,...,n$.

Then there exists an \'etale map $A \to B$, with $\Spec B$ 
an \'etale neighbourhood of $\fp$, such that for the induced map
$$\varphi_{B}^h:P_B^h := B\langle x_1,\ldots,x_n \rangle  \to \wR_{B}^h:=B\langle t_1\rangle\oplus ... \oplus B\langle t_r\rangle$$
$ \wR_B^h $ is module-finite over $P_B^h$. 
Moreover, there exists an open neighbourhood $U \subset \Spec A$ of $\fp$  such that  for each $\fq \in U$ 
we have
 $\delta_{k(\fq)} (R_\fq) \leq  \delta_{k(\fp)} (R_\fp)$.
\end{proposition}

Note that in general, $ \wR_A $ is finite over $P_A$ $\Leftrightarrow$ $\wR_A /R_A$ is finite over $R_A$ $\Leftrightarrow$  $\wR_A /\kc_A$ is finite over $R_A$,  with $\kc_A:= \Ann_{R_A} (\wR_A /R_A)$.

\begin{proof} 
Let $A^h$ be the Henselization of the local ring $A_{\fp}$ and let
$$\varphi_{A^h}^h:P_{A^h}^h:=A^h\langle x_1,\ldots,x_n \rangle  \to \wR_{A^h}^h:=A^h\langle t_1\rangle\oplus ... \oplus A^h\langle t_r\rangle$$ 
be the canonical extension of $\varphi_A$. 
$A^h$ is a local ring with maximal ideal $\fm = \fp A^h, A^h/\fm = k(\fp)$,
$P^h_{A^h}$ is a local Noetherian Henselian ring with maximal ideal $ \langle \fm, x\rangle$ and residue field $k(\fp)$, and 
$\wR_{A^h}^h$ is a Henselian semilocal ring with radical 
$\langle\fm ,x\rangle \wR_{A^h}^h= \langle \fm, t_1\rangle\oplus ... \oplus \langle \fm,t_r\rangle$.
Both are $A^h$-algebras of finite type and then
$\wR^h_{A^h}$ is also a finite type $P^h_{A^h}$-algebra (here finite type means finite type in the Henselian sense\,\footnote{ Let $A$ be a  Henselian ring. An $A$-algebra $R$ is an $A$-algebra of finite type in the Henselian sense if $R \cong A\langle y_1,...,y_s\rangle$ for suitable $y_1,...,y_s \in R$.}).
Since $\varphi_\fp$ is a parameterization,   $\varphi_{\fp,j} \neq 0$ for all $j$ and hence 
$\dim_{k(\fp)} (\wR_\fp/ \varphi_\fp(\x)\wR_\fp) =  \dim_{k(\fp)} (\wR_{A^h}^h/ \varphi_{A^h}(\langle \fm, x\rangle) \wR_{A^h}^h)< \infty$, saying that 
$\wR_{A^h}^h$ is a quasifinite and hence a finite $P_{A^h}^h$-module (by \cite[ Proposition 1.5 ]{KPP78}).

By definition of the Henselisation $A^h$, there exists an \'etale ring map $A_\fp \to D$ inducing 
$A_\fp/\fp A_\fp \cong D/\fp D$ such that $\varphi_{A^h}^h(x_i) \in B\langle t_1\rangle\oplus ... \oplus B\langle t_r\rangle$ for all $i$.
By  \cite[Proposition 28 (9)]{GP20} there is an  \'etale
map $A \to B$ with $D=B_\fp$. 
By Lemma 33 of \cite{GP20} there exists 
$\fb \in \Spec B$  such that $\fb \cap A = \fp.$ 
Let 
$\varphi_{B}^h:P_B^h := B\langle x_1,\ldots,x_n \rangle  \to \wR_{B}^h:=B\langle t_1\rangle\oplus ... \oplus B\langle t_r\rangle$
 be the map induced by
$\varphi_{A}$ and set $R_B^h := P_B^h/Ker (\varphi_B^h)$. It follows that $\wR_{B}^h$ and hence $\wR_{B}^h/R_B^h$ is a finite $P_B^h$-module. Thus  $\wR_{B}^h/R_B^h$ has a 
finite presentation  over $P_B^h$ with presentation matrix 
$T: (P_B^h)^p \to (P_B^h)^q$.

Let $P^\wedge_B := B[[x_1,\ldots,x_n]]$  be the $\x$-adic completion of $P_B^h$. We get an induced map 
$\varphi^\wedge_{B}:P_B^\wedge   \to \wR^\wedge := B[[t_1]]\oplus ... \oplus B[[ t_r]]$ and set $R^\wedge_B := P^\wedge_B/Ker (\varphi^\wedge_B)$.
Then $ \wR^\wedge/R^\wedge_B$ (the $\x$-adic completion of $\wR_{B}^h/R_B^h$) has the algebraic presentation
$T: (P^\wedge_B)^p \to (P^\wedge_B)^q$. 
 We can now apply \cite[Theorem 42]{GP20} and get that
$\dim_{k(\fc)} (\wR^\wedge_B/R^\wedge_B)\hat\otimes_B k(\fc)$ 
is semicontinuous for $\fc$ in some open neighbourhood $\tilde U \subset \Spec B$ of $\fb$. 

 Since $\pi : \Spec B \to \Spec A$ is \'etale it is open,
$U := \pi(\tilde U)$ is an open neighbourhood of $\fp$ in $\Spec A$, 
and for any $\fq \in U $ there exists a $\fc \in \tilde U$ with $\fc \cap A = \fq$.
From Lemma \cite[ Lemma 39]{GP20} we obtain 
\begin{align*}
\dim_{k(\fc)} (\wR^\wedge_B/R^\wedge_B)\hat\otimes_B k(\fc) &=
\dim_{k(\fc)} (\wR_B^h/R_B^h)\otimes^h_B k(\fc)\\
& = \dim_{k(\fq)} (\wR_A/R_A)\hat\otimes_A k(\fq).
\end{align*}
We have $(\wR_A/R_A)\hat\otimes_A k(\fq)=  \wR_\fq/R_\fq$
by Lemma \ref{lem.fiber}   and 
$\dim_{k(\fq)}  \wR_\fq/R_\fq =\delta_{k(\fq)} (R_\fq)$
by Lemma \ref{lem.prim2}. 
This implies the semicontinuity of $\delta_{k(\fq)} (R_\fq)$ on $U$. 
\end{proof}

\begin{remark}{\em In contrast to delta we cannot expect any semicontinuity for the conductor.
In Example \ref{ex.1} (2)
$c_\k(R_{\fp_\lambda})$ is not upper semicontinuous. In general $c_\k$
is also not lower semicontinuous (e.g.
for plane curves $c_\k = 2\delta_\k$ and $\delta_\k$ is  not lower semicontinuous.)  
}
\end{remark}

Two important examples are $A=\k[s]$ and $A=\Z$, treated in the examples below.

\begin{example}\label{ex.1}
{\em  We give two concrete examples with  $A=\k[s]$. In both examples 
delta is of course upper semicontinuous, but in the second example the conductor is not.

(1) Consider the family of parameterizations over $A$, 
\begin{align*}
\varphi_A & :  P_A = A[[x,y,z,u]]  \to \wR_A = A[[t]],\\
x & \mapsto t^5, y\mapsto t^6, z  \mapsto st^4 + t^8,  u\mapsto t^9,  
\end{align*} 
shortly 
$\varphi_A(s,t) = (t^5, t^6 , st^4 + t^8 , t^9)$,
with  $\varphi_A (P_A)=\k[s][[t^5, t^6 , st^4 + t^8 , t^9]] \subset \k[s][[t]]$.  For  the maximal ideals (closed points in $\Spec A$) $\fp_\lambda=\langle s-\lambda \rangle$, $\lambda \in \k$, we have thus (in the notations of Definition \ref{def.fampar}) the family
$$\varphi_{\fp_\lambda}  :   P_{\fp_\lambda}  = \k[[x,y,z,u]]  \to  \wR_{\fp_\lambda}= \k[[t]],\ \varphi_{\fp_\lambda}(t) = (t^5, t^6 , \lambda t^4 + t^8 , t^9)$$ 
of reduced curve singularities $R_{\fp_\lambda} =  \k[[x,y,z,u]]/Ker(\varphi_{\fp_\lambda}) \cong \k[[t^5, t^6 , \lambda t^4 + t^8 , t^9]] \subset \k[[t]]$. 
The parameterization is primitive and $\wR_{\fp_\lambda}$ is the normalization of $R_{\fp_\lambda}$ (Lemma \ref{lem.prim2}).
We compute easily
$\delta_\k(R_{\fp_\lambda})  = 5$ if $\lambda = 0$ and
$\delta_\k(R_{\fp_\lambda}) = 4$ if $\lambda \neq 0$. Moreover, 
$c_\k(R_{\fp_\lambda}) = 8$ for all $\lambda$.

Every closed point of $\Spec A$ is in the closure of the generic point $\eta = \langle 0 \rangle$.  Theorem \ref{thm.scont} implies that $\delta_{\k(t)} (R_\eta)  \leq \delta_\k(R_{\fp_\lambda})$,  $R_\eta=k(s)[[t^5, t^6 , st^4 + t^8 , t^9]]$,
for any $\lambda \in \k$.

 
(2) Let $\varphi_A(s,t) =(t^5,t^6,st^7+t^8,t^9)$. Then for $\fp=\langle s-\lambda \rangle$, $\lambda \in \k$, we get 
 $R_{\fp_\lambda} \cong  \k[[t^5, t^6 , \lambda t^7 + t^8 , t^9]]$ and
we compute $\delta_\k(R_{\fp_\lambda}) = 5$ for all $\lambda\in \k$. On the other hand we get $c_\k(R_{\fp_\lambda}) = 8$  if $\lambda = 0$ and $c_\k(R_{\fp_\lambda}) = 9$ if $\lambda \neq 0$.
 }
\end{example}

\begin{example}\label{ex.2}
{\em  We illustrate the difference between $R_\fp$ and   $R_A\hat\otimes_A k(\fp)$ from Remark  \ref{rm.fiber}.
Consider the previous example $\varphi_A(s,t) = (t^5, t^6 , st^4 + t^8 , t^9)$, and
use {\sc Singular} \cite{DGPS} for the following computations. 

We eliminate $t$ from $\langle x - t^5, y - t^6, z - st^4 + t^8,  u- t^9\rangle$ 
and get $R_A= \k[s][[x,y,z,u]]/I$ with $I := Ker(\varphi_A )= \langle yz-xu-x^2s,
yu-x^3,
zu-xy^2-xzs+us^2,
u^2-x^2z+xus,
y^3-x^2z+xus,
xz^2-y^2u-zus-xy^2s,
z^3-yu^2-2y^2zs-xyus-x^2ys^2-y^2s^3\rangle$.  

For $\fp_\lambda = \langle s-\lambda\rangle$, $\lambda \in \k$, we have $R_\lambda := R_A\hat\otimes_A k(\fp_\lambda)= \k[[x,y,z,u]]/(I|_{s=\lambda})$. 
If $\lambda \neq 0$
$R_{\lambda}$ is reduced and coincides with the reduced curve $R_{\fp_\lambda} $ 
considered in Example \ref{ex.1}.
But for $\lambda=0$ the ring $R_{0}$ has an embedded component (it can be computed as  $\langle u,
z^3,
yz,
y^3,
xz^2,
xy^2,
x^2z,
x^3
 \rangle$), while the reduction of  $R_{0}$ coincides with  $R_{\fp_0}$. 
}\end{example}

\begin{example}\label{ex.3}
{\em
The case $A=\Z$ is of special interest.
Consider a parameterization of an irreducible curve singularity with integer coefficients, $t \mapsto (\varphi^1(t),...,\varphi^n(t))$ with $\varphi^i(t) \in \Z[[t]]$. For $\fp \subset \Z$ a prime ideal we write
$\varphi^i_p \in \F_p[[t]]$ if  $\fp = \langle p\rangle$ with $p$ a prime number, and  $\varphi_0^i\in \Q[[t]]$ 
if $\fp = \langle 0\rangle$. 

Theorem \ref{thm.scont} implies:
If for some fixed prime $p$ the parameterization $\varphi_p$ is primitive, then $\dim_{\F_\fp} \F_p[[t]]/ \F_\fp[[\varphi_\fp^1(t),...,\varphi_\fp^n(t))]]= \delta_{\F_\fp}(R_p)$ is  finite and $\delta_{\F_\fp}(R_p) \geq \delta_{\Q}(R_0)$ as well as $\delta_{\F_\fp}(R_p) \geq  \delta_{\F_\fq}(R_q)$ for all except finitely many primes $q\in \Z$. In particular, if there exists a prime number $p$ with
$\delta_{\F_\fp}(R_p) < \infty$ then $\delta_\Q(R_0) <\infty$. Conversely,
if $\delta_\Q(R_0) $ is  finite, then $\delta_\Q(R_0) \geq\delta_{\F_q}(R_q) $
(and hence ``='') for all except finitely many prime numbers $q\in \Z$.

The same result holds for algebraic power series, i.e., if we replace 
$[[ ... ]]$ by $\langle ... \rangle$.
}
\end{example}

\begin{remark} [{\bf Computational consequences}] \label{rm.comp}
{\em
Let $A$ be an integral domain and $\eta=\langle 0 \rangle$
the generic point of $\Spec A$. Then any other point of  $\Spec A$ is in the closure of $\eta$ and it follows for a family of parameterizations
 $\varphi_A: P_A \to \wR_A$ (assumptions as in Theorem \ref{thm.scont}) that 
\begin{align}\tag{**}\label{**}
 \delta_{k(\eta)} (R_\eta)  \leq \delta_{k(\fp)}(R_{\fp})
 \end{align}
 for any $\fp \in \Spec A$.  
 
  There exist several algorithms to compute $\delta_{\k}$ for a given concrete parmeterization, eg. by computing the normalization as in \cite{GP08} by the Grauert-Remmert algorithm or by the Singh-Swanson algorithm in positive characteristic,
 or by computing a Puiseux expansion (in char 0) resp. a Hamburger Noether expansion (in char $> 0$) as in \cite {Ca80}  or by still other algorithms (see also the Manual of  {\sc Singular} \cite{DGPS}). 

\begin{itemize}
\item {\em Computing in special points:}
In applications it is often required to know the delta invariant (or at least a bound for it) at a generic point $\eta$. However the computation over 
$k(\eta) = Quot(A)$ are often extremely expansive 
(e.g. $k(\eta)=\k(t)$ for $A=\k[t]$ or $k(\eta)=\Q$  for $A=\Z$), while the computations at a special point (e.g. $\lambda \in \k$ or $p$ a prime number) are usually much faster and this can be used as a bound for delta at the generic point.
Depending on the chosen algorithm, the bound computed at a special point may be used also for early termination of the computation of delta at the generic point. 

Our result says also that delta at the generic point coincides with delta at the special points in an open dense subset of $\Spec A$ (a result which can also be obtained by the theory of Gr\"obner or standard bases)
but this open set is usually not known.
We like to emphasize, that the estimate (\ref{**}) is true for {\em every} special point, a result that cannot be deduced by Gr\"obner basis theory.
\item {\em Cutting off higher order terms during computations:} By Corollary \ref{cor.deter} the parameterization $\varphi_\fp$ of a reduced curve singularity $R_\fp$ is $(4\delta-2)$-determined, $\delta:=\delta_{k(\fp)}(R_\fp)$. 
The semicontinuity implies that for any point $\fq$ in some neighbourhood $U$ of $\fp$ (e.g. for  the generic point $\fq = \eta$)  in the power series 
$\varphi_{\fq,j}(x_i) \in k(\fq)[[t]]$ the terms of degree $>4\delta-2$ can be cut off without changing $\delta_{\varphi_\fq}$. 

It is well known, that intermediate terms of very high degree may occur during a Gr\"obner or Sagbi basis computation, which cancel at the end.
The bound $4\delta-2$ can be used to cut off higher order terms  that might appear during the computation of any $\kl$ (or $\ka$ or contact)-invariant of $\varphi_\fq$.
\end{itemize}
 }
  \end{remark}
  
\bigskip

{\bf The analytic case:}
An analogous semicontinuity result for complex analytic families (with a much easier proof) may be of independent interest. 
A morphism of analytic $\C$-algebras 
$\varphi :  \C\{x_1,...,x_n\} \to \C\{t_1\}\oplus ... \oplus  \C\{t_r\}$ 
is called a parameterization of the analytic curve singularity 
$Im (\varphi)$ with $r$ branches if it satisfies the conditions as in  Definition \ref{def.param}, with formal power series replaced by convergent ones. The notion of a primitive parameterization carries over  to analytic $\C$-algebras 
and Lemma \ref{lem.prim2} holds also in the analytic case (in particular, 
$\varphi$ is primitive $\Leftrightarrow$  $\varphi$ is the normalization map 
$\Leftrightarrow$ 
$\dim_\C \C\{t_1\}\oplus ... \oplus  \C\{t_r\} / Im(\varphi) < \infty).$

Let $(S,0)$  be a complex analytic germ and 
$$\varphi_S: \ko_{S,0}\{x_1,...,x_n\} \to \ko_{S,0}\{t_1\}\oplus ... \oplus  \ko_{S,0}\{t_r\}$$ 
an analytic $ \ko_{S,0}$-algebra morphism. 
For $S$ a representative of $(S,0)$  and $s\in S$ let 
\begin{align*}
 \varphi_s  = \varphi_S \otimes_{\ko_{S,s}} \C: &
\ \C\{x_1,...,x_n\} \to \C\{t_1\}\oplus ... \oplus  \C\{t_r\}\\
 \varphi_{s,j}: & \ \C\{x_1,...,x_n\} \to \C\{t_j\},\  j=1,...,r,
 \end{align*}
be the induced maps. 
Then $\ko_{C_s,0}:= \C\{x_1,...,x_n\} /Ker ( \varphi_s ) \cong
Im (\varphi_s) = \C \{\varphi_{s}(x_1),...,\varphi_{s}(x_n)\} \subset \C\{t_1\}\oplus ... \oplus  \C\{t_r\}$ is the local ring of the germ 
$(C_s,0) =V(Ker(\varphi_{s})) \subset (\C^n,0)$.

\begin{proposition}\label{prop.scont-an} 
With the above notations assume that $\varphi_0$ is a primitive parameterization of the reduced curve singularity  with $r$ branches 
$(C_0,0)$. Then,
 for $s$ in a sufficiently small neighbourhood of $0\in S$, $(C_s,0)$ is a reduced curve singularity with with $r$ branches satisfying
$$\dim_\C \C\{t_1\}\oplus ... \oplus  \C\{t_r\} / \C \{\varphi_{s}(x_1),...,\varphi_{s}(x_n)\}  = \delta_{\mathbb C} (\ko_{C_s,0}) < \infty$$
 and 
$\delta_{\mathbb C} (\ko_{C_s,0}) \leq \delta_{\mathbb C} (\ko_{C_0,0})$.
\end{proposition}

\begin{proof} That $(C_s,0)$ is a reduced curve singularity with $r$ branches follows as in the formal case (Remark \ref{rm.cond}). Since  primitive implies (Lemma \ref{lem.prim2})
$\dim_\C Coker (\varphi_0) < \infty$ the Weierstrass Finiteness Theorem (\cite [Theorem 1.10]{GLS07}) implies that  $Coker (\varphi_S)$ is a 
finite and hence a coherent $\ko_S$-module (by (\cite [Theorems 1.66 and 1.67]{GLS07}). 
The result follows since 
$Coker (\varphi_S) \otimes_{\ko_{S,s}} \C = Coker (\varphi_s)$ and 
$\dim_\C Coker (\varphi_s) = \delta_{\mathbb C} (\ko_{C_s,0})$ as in Lemma \ref{lem.prim2}. 
\end{proof}

\begin{remark}\label{rm.scont-an}
{\em 
Let  $\k$ be a real-valued field with a valuation and 
$A$ an  analytic $\k$-algebra, i.e.   $A\cong \k\{y\}/I$, $y = ( y_1,\ldots,y_s)$, $I$ an ideal, and  $\k\{y\}$ the convergent power series ring over $\k$ (cf. \cite{GLS07}).
 Then 
 small  $\varepsilon$-neighbourhoods in $\k^s$ are defined and 
 the same semicontinuity result for $\delta_{\k}$ as in the complex case
 (Proposition \ref{prop.scont-an}) holds more generally
for morphisms of analytic $\k$-algebras
$$\varphi_A: A\{x_1,...,x_n\} \to A\{t_1\}\oplus ... \oplus  A\{t_r\}.$$
The proof is basically the same as in the complex-analytic case. The Weierstrass Finiteness Theorem holds also in this case (\cite [Theorem 1.10]{GLS07}) and instead of the coherence theorem, one can use \cite [Lemma 1]{GP20}.

Examples of real-valued fields with valuation are:
  
- Any field  $\k$ with the trivial valuation, then $\k\{y\} =\k[[y]]$ is the formal power series ring. Finite fields have only the trivial valuation.
  
- $\k \in \{\Q, \R, \C\}$  with the absolute value as (Archimedian) valuation, then  $\k\{y\}$ is the usual convergent power series ring.  

- Another example is $\k=\mathbb Q$ with the $p$-adic valuation\footnote{ Write $q\in \mathbb Q$ as $q=p^a\cdot \frac{r}{s}$ with $r,s$ coprime to $p$, then the $p$-adic valuation is $|q|_p=p^{-a}$. }, a non-Archimedian valuation on $\mathbb Q$.
 
- The completion of a real-valued field is again a real-valued field. 
The completion of the field $\Q$ with respect to the ordinary absolute value
is $\R$, while the completion of $\mathbb Q$ with respect to the $p$-adic valuation is the field of $p$-adic numbers $\Q_p$\footnote { $\Q_p = \{\sum_{i=n}^\infty a_ip^i | n\in\Z, a_i \in \{0,...,p-1\} \} $. 
$\Q \subset \Q_p$ since $-1=\sum_0^\infty (p-1)p^i $. }.
 
The convergent power series ring over a real-valued field has similar good properties as the formal power series ring. The most important
property is the Weierstrass Finiteness Theorem (\cite [Theorem 1.10]{GLS07}). Of interest may be that $\k\{y\}$ is excellent\footnote{ For a definition of excellence see  \cite[15.51]{Stack}.} if and only if the completion of $\k$ with respect to the valuation is separable over $\k$ (cf. \cite{Sch82}).
}
\end{remark}


\noindent
\footnotesize
{Gert-Martin Greuel and Gerhard Pfister\\
University of Kaiserslautern\\
Department of Mathematics\\
Erwin-Schroedinger Str.\\
67663 Kaiserslautern\\
Germany\\
e-mail: greuel@mathematik.uni-kl.de\\
e-mail: pfister@mathematik.uni-kl.de}

\end{document}